\documentclass[10pt]{article}

\usepackage{titlefoot}
\usepackage[small]{titlesec}
\usepackage[small,it]{caption}

\usepackage{amsmath}
\usepackage{amsthm}
\usepackage{amssymb}

\usepackage{tikz}
	\usetikzlibrary{shapes}
	\usetikzlibrary{calc}

\usepackage{listings}
\makeatletter
\AtBeginDocument{%
  \let\c@figure\c@lstlisting
  
  \let\c@table\c@lstlisting
  
  \let\ftype@lstlisting\ftype@figure % give the floats the same precedence
  \let\ftype@table\ftype@figure % give the floats the same precedence
}
\makeatother

\usepackage{color}
\definecolor{lightgray}{rgb}{0.8, 0.8, 0.8}
\definecolor{darkgray}{rgb}{0.7, 0.7, 0.7}
\definecolor{darkblue}{rgb}{0, 0, .4}

\usepackage[bookmarks]{hyperref}
\hypersetup{
        colorlinks=true,
        linkcolor=darkblue,
        anchorcolor=darkblue,
        citecolor=darkblue,
        urlcolor=darkblue,
        pdfpagemode=UseThumbs,
        pdftitle={Containing all permutations},
        pdfsubject={Combinatorics},
        pdfauthor={Engen and Vatter},
}

\newcounter{todocounter}

\theoremstyle{plain}
\newtheorem{theorem}{Theorem}[section]
\newtheorem{proposition}[theorem]{Proposition}
\newtheorem{observation}[theorem]{Observation}

\theoremstyle{definition}

\makeatletter
\newfont{\footsc}{cmcsc10 at 8truept}
\newfont{\footbf}{cmbx10 at 8truept}
\newfont{\footrm}{cmr10 at 10truept}
\renewenvironment{abstract}%
                {
                  \begin{list}{}%
                     {\setlength{\rightmargin}{1in}%
                      \setlength{\leftmargin}{1in}}%
                   \item[]\ignorespaces\begin{small}}%
                 {\end{small}\unskip\end{list}}

\newcommand{\st}{\::\:}
\newcommand{\Av}{\operatorname{Av}}
\newcommand{\C}{\mathcal{C}}

\newcommand\mybullet{\raisebox{-5pt}{\normalsize \ensuremath{\bullet}}}
\newcommand\mycirc{\raisebox{-5pt}{\normalsize \ensuremath{\circ}}}

\makeatletter
\def\absdot{\@ifnextchar[{\@absdotlabel}{\@absdotnolabel}}
	\def\@absdotlabel[#1]#2{%
		\node at #2 {\normalsize \mybullet};
		\node at #2 [below=2pt] {\ensuremath{#1}};
	}
	\def\@absdotnolabel#1{%
		\node at #1 {\normalsize \mybullet};
	}
\def\absdothollow{\@ifnextchar[{\@absdothollowlabel}{\@absdothollownolabel}}
	\def\@absdothollowlabel[#1]#2{%
		\node at #2 {\normalsize \textcolor{white}{\mybullet}};
		\node at #2 {\normalsize \mycirc};
		\node at #2 [below=2pt] {\ensuremath{#1}};
	}
	\def\@absdothollownolabel#1{%
		\node at #1 {\normalsize \textcolor{white}{\mybullet}};
		\node at #1 {\normalsize \mycirc};
	}
\makeatother

\newcommand{\plotperm}[1]{
	\foreach \j [count=\i] in {#1} {
		\absdot{(\i,\j)};
	};
}

\newcommand{\plotpermbox}[4]{
	\draw [darkgray, thick, line cap=round]
		({#1-0.5}, {#2-0.5}) rectangle ({#3+0.5}, {#4+0.5});
}

\datefoot{\today}
\amssubj{05A05 (primary), 68R15, 05D99 (secondary)}
\keywords{universal permutation, superpattern, superpermutation}

\newpagestyle{main}[\small]{
        \headrule
        \sethead[\usepage][][]
        {\sc Containing All Permutations}{}{\usepage}}

\title{\sc Containing All Permutations}
\author{Michael Engen and Vincent Vatter\\[-0.25ex]
\small Department of Mathematics\\[-0.5ex]
\small University of Florida\\[-0.5ex]
\small Gainesville, Florida USA\\[-1.5ex]
}

\titleformat{\section}
        {\large\sc}
        {\thesection.}{1em}{}

\date{}

\begin{document}

\maketitle

\begin{abstract}
Numerous versions of the question ``what is the shortest object containing all permutations of a given length?'' have been asked over the past fifty years: by Karp (via Knuth) in 1972; by Chung, Diaconis, and Graham in 1992; by Ashlock and Tillotson in 1993; and by Arratia in 1999. The large variety of questions of this form, which have previously been considered in isolation, stands in stark contrast to the dearth of answers. We survey and synthesize these questions and their partial answers, introduce infinitely more related questions, and then establish an improved upper bound for one of these questions.
\end{abstract}

\pagestyle{main}

\section{Introduction}
\label{sec-intro}

What is the shortest object containing all permutations of length $n$? As we shall describe, there are a variety of such problems, going by an assortment of names including superpatterns and superpermutations. Throughout, we call all such problems \emph{universal permutation problems}. The diversity of these problems stems from the multiple possible definitions of the terms involved. 

To state these problems, it is necessary to view permutations as words. A \emph{word} is simply a finite sequence of \emph{letters} or \emph{entries} drawn from some \emph{alphabet}. The \emph{length} of the word $w$, denoted $|w|$ throughout, is its number of letters, and if $w$ is a word of length at least $i$, then we denote by $w(i)$ its $i$th letter. From this viewpoint, a \emph{permutation} of length $n$ is a word consisting of the letters $[n]=\{1,2,\dots,n\}$, each occurring precisely once. Permutations are thus a special type of words over the positive integers $\mathbb{P}$. Two words $u,v\in\mathbb{P}^n$ (that is, both of length $n$, with positive integer letters) are \emph{order-isomorphic} if, for all indices $i,j\in[n]$, we have
\[
	u(i)>u(j) \iff v(i)>v(j).
\]

In all universal permutation problems considered here, the object that is to contain all permutations of length $n$, called the \emph{universal} object, is a word, but there are two different types of containment. Sometimes we insist that the word $w$ contain each such permutation $\pi$ as a contiguous subsequence, or \emph{factor}, by which we mean that $w$ can be expressed as a concatenation $w=upv$ where the word $p$ is order-isomorphic to $\pi$. At other times we merely insist that $w$ contain each such permutation $\pi$ as a \emph{subsequence}, by which we mean that there are indices $1\le i_1<i_2<\cdots<i_n\le|w|$ so that the word $p=w(i_1)w(i_2)\cdots w(i_n)$ is order-isomorphic to $\pi$.

These notions of containment give rise to two different universal permutation problems. To obtain infinitely many, we vary the size of the alphabet that the letters of the universal word $w$ can be drawn from. In the strictest form, we insist that $w$ be a word over the alphabet $[n]$, meaning that $w$ is only allowed the symbols of the permutations it must contain. In this case, the notion of order-isomorphism reduces to equality: a word $p\in[n]^n$ is order-isomorphic to a permutation $\pi$ of length $n$ if and only if $p=\pi$. At the other end of the spectrum, we allow the letters of $w$ to be arbitrary positive integers. Between these extremes, another interesting case stands out: when the alphabet is $[n+1]$, thus allowing the universal word one more symbol than the permutations it must contain. Table~\ref{tab-six-problems} displays the best upper bounds established to date for the six versions of this question that have garnered the most interest. In the case of the rightmost two cells of the upper row, these upper bounds are known to be the actual answers.

\begin{table}
\caption{Current best upper bounds for the lengths of the shortest universal words in six flavors of universal permutation problems (for large $n$).}
\vspace{-\bigskipamount}
\[
\begin{array}{r|ccccc}
	&\text{words over $[n]$}
	&
	&\text{words over $[n+1]$}
	&
	&\text{words over $\mathbb{P}$}\\[4pt]
	\hline\\[-6pt]
	\text{factor}
		%&\displaystyle\left(\sum_{k=1}^n k!\right)-1
		&n!+(n-1)!+(n-2)!
		&
		&n!+n-1
		&
		&n!+n-1
	\\
		&\quad\quad\ +(n-3)!+n-3
	\\[16pt]
	\text{subsequence}
		&\displaystyle\left\lceil n^2-\frac{7}{3}n+\frac{19}{3}\right\rceil
		&
		&\displaystyle\frac{n^2+n}{2}
		&
		&\displaystyle\left\lceil\frac{n^2+1}{2}\right\rceil
\end{array}
\]
\vspace{-\medskipamount}
\label{tab-six-problems}
\end{table}

The bounds shown in Table~\ref{tab-six-problems} weakly decrease as we move from left to right (a word over $[n]$ is also a word over $[n+1]$, which is also a word over $\mathbb{P}$) and also as we go from top to bottom (factors are also subsequences). Another notable feature of this table is that the lengths of the shortest universal words over the alphabet $[n]$ seem to be significantly greater than the lengths of the shortest universal words over the alphabet $[n+1]$, whose lengths seem to be either equal to or close to those of the shortest universal words over the largest possible alphabet, $\mathbb{P}$.

We remark that some research in this area has sought a universal \emph{permutation} instead of a universal word, but this is in fact equivalent to finding a universal word over $\mathbb{P}$, as we briefly explain. The word $u\in\mathbb{P}^n$ is \emph{order-homomorphic} to the word $v\in\mathbb{P}^n$ if, for all indices $i,j\in[n]$, we  have
\[
	u(i)>u(j) \implies v(i)>v(j).
\]
Less formally, if $u$ is order-homomorphic to $v$, then all strict inequalities between entries of $u$ also hold between the corresponding entries of $v$, but equalities between entries of $u$ may be broken in $v$. It is clear that every word over $\mathbb{P}$ is order-homomorphic to at least one permutation (one simply needs to ``break ties'' among the letters of the word), and it follows that if $u$ contains the permutation $\pi$ (as a factor or subsequence) and $u$ is order-homomorphic to $v$, then $v$ also contains $\pi$ (in the same sense---indeed, in the same indices---that $u$ contains it). As every permutation is also a word over $\mathbb{P}$, it follows finding a universal word, in either the factor or subsequence setting, is equivalent to finding a universal word over $\mathbb{P}$.

Each of the subsequent five sections of this paper is devoted to the examination of one of the cells of Table~\ref{tab-six-problems} (except for Section~\ref{sec-factors-n+1}, which considers both the upper-center and upper-right cells). While the results described in Sections~\ref{sec-factors-n}--\ref{sec-seq-n+1} are previously known, the results of Section~\ref{sec-subseq-P} appear for the first time here. In the final section, we briefly describe some further variations on universal permutation problems.

\section{As Factors, Over $[n]$}
\label{sec-factors-n}

The case in the upper-left of Table~\ref{tab-six-problems} dates to a 1993 paper of Ashlock and Tillotson~\cite{ashlock:construction-of:} and can be restated as follows.
\begin{quote}
	What is the length of the shortest word over the alphabet $[n]$ that contains each permutation of length $n$ as a factor?
\end{quote}
This version of the universal permutation problem has recently attracted a surprising amount of attention, including an article in \emph{The Verge}~\cite{griggs:an-anonymous-4c:} and two in \emph{Quanta Magazine}~\cite{honner:unscrambling-th:,klarreich:mystery-math-wh:}, and investigations are very much ongoing.

We call a word over the alphabet $[n]$ that contains all permutations of length $n$ as factors an \emph{$n$-superpermutation}. A (not particularly good) lower bound on the length of $n$-superpermutations is easy to establish by observing that every word $w$ has at most $|w|-n+1$ many factors of length $n$.

\begin{observation}
\label{obs-number-of-factors}
Every $n$-superpermutation has length at least $n!+n-1$.
\end{observation}

In the cases of $n=1$ and $n=2$, the shortest $n$-superpermutations are easy to find. The word $1$ meets the demands for $n=1$ and the word $121$ is as short as possible for $n=2$. The shortest $3$-superpermutation has length $9$---one more than the lower bound above, but may be shown to be optimal with a slightly more delicate argument, which we now present. First, there is a word of length $9$,
\[
	123121321,
\]
that contains all permutations of length $3$ as factors. Now suppose that the word $w$ over the alphabet $[3]$ contains all permutations of length $3$ as factors. We say that the letter $w(i)$ is \emph{wasted} if the factor $w(i-2)w(i-1)w(i)$ is not equal to a new permutation of length $3$---either because not all of the letters are defined, or because it contains a repeated letter, or because that permutation occurs earlier in $w$. As each nonwasted letter corresponds to the first occurrence of a permutation, we have 
\[
	|w| 
	= 
	3! + (\text{\# of wasted letters in $w$}).
\] 
Clearly the first two letters of $w$ are wasted. If $w$ contains an additional wasted letter, then its length must be at least $9$. Suppose then that $w$ does not contain any additional wasted letters. Thus each of the factors
\[
	w(1)w(2)w(3),\ 
	w(2)w(3)w(4),\ 
	w(3)w(4)w(5),\ 
	\text{and}\ 
	w(4)w(5)w(6)
\]
must be equal to different permutations. However, the only way for these factors to be equal to permutations at all is to have $w(4)=w(1)$, $w(5)=w(2)$, and $w(6)=w(3)$, and this implies that $w(4)w(5)w(6)=w(1)w(2)w(3)$, a contradiction.

Computations by hand become more difficult at $n=4$, but we invite the reader to check that the word
\[
	123412314231243121342132413214321
\]
of length $33$ is a $4$-superpermutation, and that no shorter word suffices.

As Ashlock and Tillotson~\cite{ashlock:construction-of:} noticed, the lengths of these superpermutations are, respectively,
\begin{eqnarray*}
1!&=&1,\\
2!+1!&=&3,\\
3!+2!+1!&=&9,\text{ and}\\
4!+3!+2!+1!&=&33.
\end{eqnarray*}
They also gave a recursive construction establishing the following result.

\begin{proposition}[Ashlock and Tillotson~{\cite[Theorem 3 and Lemma 5]{ashlock:construction-of:}}]
\label{prop-factor-[n]-upper-bound}
If there is an $(n-1)$-superpermutation of length $m$, then there is an $n$-superpermutation of length $n!+m$.
\end{proposition}

Proposition~\ref{prop-factor-[n]-upper-bound} guarantees an $n$-superpermutation of length at most $n!+\cdots+2!+1!$. Given this construction and the lower bounds they had been able to compute, Ashlock and Tillotson made the natural conjecture that the shortest $n$-superpermutation has length $n!+\cdots+2!+1!$ for all $n$. They further conjectured that all of the shortest $n$-superpermutations were unique up to the relabeling of their letters.

For about twenty years, very little progress seemed to have been made on these conjectures, although they were rediscovered many times on Internet forums such as MathExchange and StackOverflow (references to some of these rediscoveries are given in Johnston's article~\cite{johnston:non-uniqueness-:}). Then, in 2013, Johnston~\cite{johnston:non-uniqueness-:} constructed multiple distinct $n$-superpermutations of length $n!+\cdots+2!+1!$ for all $n\ge 5$,  proving that at least one of Ashlock and Tillotson's two conjectures must be false, although giving no hint as to which one. A year later, Benjamin Chaffin verified the $n=5$ case of the length conjecture by computer (see Johnston's blog post~\cite{johnston:all-minimal-sup:} for details), showing that no word of length less than $153=5!+4!+3!+2!+1!$ is a $5$-superpermutation. This showed, via Johnston's constructions, that Ashlock and Tillotson's uniqueness conjecture was certainly false, although their length conjecture might still have held.

The next case of the length conjecture to be verified would be $n=6$, where the conjectured shortest length was $6!+5!+4!+3!+2!+1!=873$. However, only weeks after Chaffin's verification of the length conjecture for $n=5$, Houston~\cite{houston:tackling-the-mi:}---by viewing the problem as an instance of the traveling salesman problem---found a $6$-superpermutation of length only $872$. 

Whether this is the shortest $6$-superpermutation is the focus of an ongoing distributed computing project at 
\[
	\text{\url{www.supermutations.net}.}
\]
Regardless of the outcome of that project, the $6$-superpermutation of length $872$ and Proposition~\ref{prop-factor-[n]-upper-bound} reduce the upper bound on the length of the shortest $n$-superpermutation to {$n!+\cdots+3!+2!$} for all $n\ge 6$.

% Removed:
% 	and searching for solutions using Helsgaun's implementation~\cite{helsgaun:an-effective-im:} of the Lin--Kernighan heuristics~\cite{lin:an-effective-he:} for such problems

After breaking the length conjecture of Aslock and Tillotson in 2014, Houston created a Google discussion group called Superpermutators, where those interested in the problem could work on it in a loose and Polymath-esque manner, and most of the subsequent research mentioned here has been communicated there.

The next breakthrough was made shortly after John Baez tweeted about Houston's construction in September 2018. This tweet caused Greg Egan, who is known for his science fiction novels (coincidently including one entitled \emph{Permutation City}~\cite{egan:permutation-cit:}), to become interested in the problem. Egan found inspiration in an unpublished manuscript of Williams~\cite{williams:hamiltonicity-o:}. In that paper, Williams showed how to construct Hamiltonian paths and cycles in the Cayley graph on the symmetric group $S_n$ generated by the two permutations denoted by $(12\cdots n)$ and $(12)$ in cycle notation (see Sawada and Williams~\cite{sawada:a-hamilton-path:} for a published, streamlined construction). Williams's construction had solved a forty year-old conjecture of Nijenhuis and Wilf~\cite{nijenhuis:combinatorial-a:} (later included by Knuth as an exercise with a difficulty rating of $48/50$ in Volume 4A of the \emph{Art of Computer Programming}~\cite[Problem 71 of Section 7.2.1.2]{knuth:the-art-of-comp:4a}), and, in October 2018, Egan showed how it could be adapted to prove the following.

\begin{theorem}[Egan~\cite{egan:superpermutatio:}]
\label{thm-egan-upper-bound}
For all $n\ge 4$, there is an $n$-superpermutation of length at most
\[
		n!+(n-1)!+(n-2)!+(n-3)!+n-3.
\]
\end{theorem}

For $n=6$, the construction of Theorem~\ref{thm-egan-upper-bound} is worse than Houston's (Theorem~\ref{thm-egan-upper-bound} gives a $6$-superpermutation of length $873$), but for $n\ge 7$ this bound is strictly less than the bound of $n!+\cdots+3!+2!$ implied by Houston's construction and Proposition~\ref{prop-factor-[n]-upper-bound}.

The efforts described above yield upper bounds. For lower bounds, Ashlock and Tillotson improved on Observation~\ref{obs-number-of-factors} by focusing on wasted letters as we did earlier in the $n=3$ case. For general $n$, we say that the letter $w(i)$ is \emph{wasted} if the factor
\[
	w(i-n+1)w(i-n+2)\cdots w(i)
\]
is either not a permutation of length $n$, or occurs earlier in $w$. The crucial observation is that if neither $w(i)$ nor $w(i+1)$ are wasted letters, then the permutations ending at those letters are cyclic rotations of each other. The $n!$ permutations of length $n$ can be partitioned into $(n-1)!$ disjoint cyclic classes, where the \emph{cyclic class} of the permutation $\pi$ consists of all of its cyclic rotations. For example, the cyclic class of the permutation $12345$ is
\[
	\{12345, 23451, 34512, 45123, 51234\}.
\]
Our reasoning above implies that upon completing a cyclic class (having visited all of its members), the next letter in the word (if there is one) must be wasted. Any $n$-superpermutation must complete all $(n-1)!$ cyclic classes, and thus doing so requires at least $(n-1)!-1$ wasted letters. Together with the $n-1$ letters at the beginning of $w$, which are trivially wasted, we obtain the following result.

\begin{proposition}[Ashlock and Tillotson~{\cite[proof of Theorem 18]{ashlock:construction-of:}}]
\label{prop-factor-[n]-lower-bound}
For all $n\ge 1$, every $n$-superpermutation has length at least
\[
	n!+(n-1)!+n-2.
\]
\end{proposition}

At least since a 2013 blog post of Johnston~\cite{johnston:the-minimal-sup:}, it had been known that there was an argument (on a website devoted to anime) claiming to improve on the lower bound provided by Proposition~\ref{prop-factor-[n]-lower-bound}. However, the argument was far from what most mathematicians would consider a proof, and there had been no efforts to make it into one, in part because the claimed lower bound was so far from what was thought to be the correct answer at the time. However, Egan's breakthrough quickly inspired several participants of the Superpermutators group to re-examine the argument. In the process, it was realized not only that the argument was correct, but that it did not originate on the anime website where Johnston had found it. Instead, it had been copied there from a series of anonymous posts in 2011 on the somewhat-notorious Internet forum 4chan.

The crux of the argument is an idea that we call a trajectory (though the original proof called it a $2$-loop). The proof of Proposition~\ref{prop-factor-[n]-lower-bound} suggests that in building an $n$-superpermutation, one might try to complete an entire cyclic class, then waste a single letter to enter a new cyclic class, and so on. For example, in the $n=5$ case, suppose we visit the cyclic class of $12345$ in order,
\[
	12345, 23451, 34512, 45123, 51234.
\]
Once we have come to the $4$ of $51234$, there is a unique way to waste a single letter to move to a different cyclic class; this is to append the letters $15$, and doing so moves us to the cyclic class of the permutation $23415$. It would then be natural to complete this cyclic class, by visiting the permutations
\[
	23415, 34152, 41523, 15234, 52341.
\]
After that it would again be natural to waste a letter to traverse to the cyclic class of $34125$ and complete that class by visiting the permutations
\[
	34125, 41253, 12534, 25341, 53412.
\]
Finally, by wasting a letter to enter and complete the cyclic class of $41235$, we would encounter the permutations
\[
	41235, 12354, 23541, 35412, 54123
\]
in that order. However, from that point there would be no way to waste a single letter to enter a new cyclic class; by appending a $5$ we would cycle back to $41235$, while appending a $45$ would return us to $12345$. In general, by following this procedure one would visit ${n-1}$ cyclic classes before reaching a point where wasting a single letter would either cause us to stay in the same cyclic class or to return to the initial permutation. We define the \emph{trajectory} of the permutation $\pi$ of length $n$ to consist of the sequence of $n(n-1)$ permutations visited by following this procedure starting at $\pi$, and thus the above sequence of permutations is the trajectory of $12345$. We caution the reader that trajectories do \emph{not} partition the set of permutations; while $23451$ lies in the trajectory of $12345$, the trajectories of $23451$ and of $12345$ contain different sets of permutations.

As we read a superpermutation from left to right, we keep track of which trajectory we are on. We begin on the trajectory of the first permutation we see in the word. After that, we say that we \emph{change} trajectories whenever the word deviates from the above pattern of traversing an entire cyclic class, wasting a letter, traversing an entire cyclic class, etc., and the trajectory we change \emph{to} is the trajectory of the first permutation encountered after a change of trajectories. Changing trajectories obviously requires at least one wasted letter because one must at least change cyclic classes to change trajectories. We view the wasted letter immediately before entering the new trajectory (that is, encountering a new permutation) as the letter wasted to change trajectories. As each trajectory contains $n(n-1)$ permutations, any $n$-superpermutation must change trajectory at least $(n-2)!-1$ times, and doing so requires at least $(n-2)!-1$ wasted letters.

To improve on Proposition~\ref{prop-factor-[n]-lower-bound}, we now argue as follows. As in the proof of Proposition~\ref{prop-factor-[n]-lower-bound}, any $n$-superpermutation must complete all $(n-1)!$ cyclic classes, and doing so requires at least $(n-1)!-1$ wasted letters. We view the letter wasted immediately after completing a cyclic class as the letter wasted to leave a completed cyclic class. For example, suppose that our word begins with the prefix
\[
	123451234\underline{1}523\underline{1}4.
\]
Thus we begin on the trajectory of $12345$. The next four letters, $1234$, complete the cyclic class of $12345$. The letter immediately after that (the first $\underline{1}$ above) is wasted to leave that completed cyclic class. We then visit the permutations $23415$, $34152$, and $41523$ in that order before wasting another letter (the second $\underline{1}$ above) to change trajectories.

Finally, we note that the letters wasted to complete cyclic classes and those wasted to change trajectory must be distinct---indeed, this claim amounts to saying that when one has completed a cyclic class, wasting a single letter does not change trajectories. This completes the proof of the following result.

\begin{theorem}[Anonymous 4chan poster]
\label{thm-4chan-lower-bound}
For all $n\ge 1$, every $n$-superpermutation has length at least
\[
	n!+(n-1)!+(n-2)!+n-3.
\]
\end{theorem}

Houston has shown (in the Superpermutators group) that the bound in Theorem~\ref{thm-4chan-lower-bound} can be increased by $1$. For general $n$, Theorem~\ref{thm-egan-upper-bound} and this improvement to Theorem~\ref{thm-4chan-lower-bound} are the best results established so far. There had been some hope in the Superpermutators group that perhaps Egan's construction could be made one letter shorter for $n\ge 7$, while the lower bound could be increased by $(n-3)!-1$, so that the two met at
\[
	n!+(n-1)!+(n-2)!+(n-3)!+n-4,
\]
but this has also been shown to be false in the $n=7$ case. In this case, the original length conjecture of Ashlock and Tillotson suggested that the length of the shortest $7$-superpermutation should be $7!+6!+5!+4!+3!+2!+1!=5913$, while Egan's Theorem~\ref{thm-egan-upper-bound} gives a $7$-superpermutation of length $7!+6!+5!+4!+4=5908$. In February 2019, Bogdan Coanda made several theoretical improvements to the computer search for superpermutations and used these to find a $7$-superpermutation of length $7!+6!+5!+4!+3=5907$, thus matching the wishful thinking above. (Continuing the tradition of ``publishing'' progress on this problem in unorthodox places, Coanda announced his construction pseudonymously in the comment section of a YouTube video~\cite{parker:superpermutatio:} about the problem.) Shortly thereafter, Egan and Houston modified Coanda's approach to construct a $7$-superpermutation of length $7!+6!+5!+4!+2=5906$.

\section{As Factors, Over $[n+1]$ and $\mathbb{P}$}
\label{sec-factors-n+1}

In moving from the previous universal permutation problem to this one, we see for the first of two times the dramatic effect of adding a letter to the alphabet. Not only does the addition of a single letter seem to significantly shorten the universal words, but it changes the problem from one that remains wide open to one solved a decade ago.

A \emph{de Bruijn word} of order $n$ over the alphabet $[k]$ is a word $w$ of length $k^n$ such that every word in $[k]^n$ occurs exactly once as a \emph{cyclic factor} in $w$, or equivalently, every such word occurs exactly once as a factor in the longer word
\[
	w(1)w(2)\cdots w(k^n)\ w(1)w(2)\cdots w(n-1).
\]
These words were (mis)named for de Bruijn (see \cite{bruijn:acknowledgement:}) because in addition to establishing that such words exist, he showed that there are precisely $(k!)^{k^{n-1}}/k^n$ of them. An example of a de Bruijn word, written cyclically, is shown on the left of Figure~\ref{fig-debruijn}.

\begin{figure}
\begin{footnotesize}
\begin{center}
	\begin{tikzpicture}[scale=0.375, baseline=(current bounding box.south)]
		\foreach \val [count = \i] in {1,1,1,2,2,2,3,3,3,1,2,3,1,3,3,2,2,1,1,3,2,3,2,1,2,1,3} {
			\draw ({90-(\i-1)*360/27}: 4) node{\val};
		}
		\draw [->] (0,1) arc (90:-180: 1);
	\end{tikzpicture}
\quad\quad\quad\quad
	\begin{tikzpicture}[scale=0.375, baseline=(current bounding box.south)]
		\foreach \val [count = \i] in {1,2,3,4,1,2,5,3,4,1,5,3,2,1,4,5,3,2,4,1,3,2,5,4} {
			\draw ({90-(\i-1)*360/24}: 4) node{\val};
		}
		\draw [->] (0,1) arc (90:-180: 1);
	\end{tikzpicture}
\vspace{0.15in}
\caption{On the left, a de Bruijn word of order $3$ over the alphabet $[3]$. On the right, a universal cycle for the permutations of length $4$ over the alphabet $[5]$.}
\label{fig-debruijn}
\end{center}
\end{footnotesize}
\end{figure}
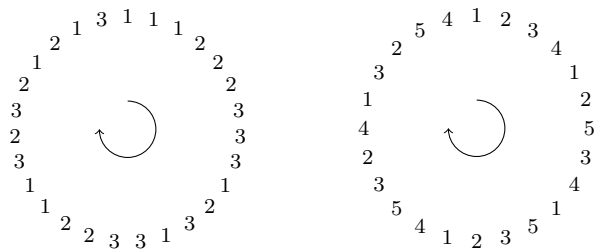

In their highly influential 1992 paper, Chung, Diaconis, and Graham~\cite{chung:universal-cycle:} explored generalizations of de Bruijn words to other types of objects, including permutations. (In fact, Diaconis and Graham~\cite[Chapter 4]{diaconis:magical-mathema:} state that their motivation was a magic trick.) As they defined it, a \emph{universal cycle} (frequently shortened to \emph{ucycle}) for the permutations of length $n$ would be a word $w$ of length $n!$ (over some alphabet) such that every permutation of length $n$ is order-isomorphic to a cyclic factor of $w$, or equivalently, to a factor of the slightly longer word $w(1)w(2)\cdots w(n!)\ w(1)w(2)\cdots w(n-1)$. An example of a universal cycle over $[5]$, written cyclically, for the permutations of length $4$ is shown on the right of Figure~\ref{fig-debruijn}.

If such a universal cycle $w$ were to exist (which was the question they were interested in, leaving enumerative concerns for later), then the word
\[
	w(1)w(2)\cdots w(n!)\ w(1)w(2)\cdots w(n-1)
\]
would be, in our terms, a shortest possible answer to the universal permutation problem for factors over the alphabet $\mathbb{P}$. In this way, their universal cycle of length $4!=24$ for the permutations of length $4$ shown on the right of Figure~\ref{fig-debruijn} is converted (starting at noon and proceeding clockwise) into the universal word
\[
	123412534153214532413254\ 123
\]
of length $4!+4-1=27$. Thus, together with the trivial lower bound of $n! + n - 1$ noted in Observation~\ref{obs-number-of-factors}, the answer to the question posed in the upper-righthand cell of Table~\ref{tab-six-problems} is implied by the following result.

\begin{theorem}[Chung, Diaconis, and Graham~\cite{chung:universal-cycle:}]
\label{thm-univ-factor-P}
For all positive integers $n$, there is a universal cycle over the alphabet $[6n]$ for the permutations of length $n$.
\end{theorem}

Chung, Diaconis, and Graham left open the question of whether the alphabet $[6n]$ could be shrunk. Proposition~\ref{prop-factor-[n]-lower-bound} shows that, for $n\ge 3$, there cannot be a universal cycle over the alphabet $[n]$ for the permutations of length $n$. Therefore the result below, established by Johnson in 2009, is best possible.

\begin{theorem}[Johnson~\cite{johnson:universal-cycle:}]
\label{thm-univ-factor-n+1}
For all positive integers $n$, there is a universal cycle over the alphabet $[n+1]$ for the permutations of length $n$.
\end{theorem}

In terms of universal permutation problems, Theorem~\ref{thm-univ-factor-n+1} establishes that there is a word of length $n!+n-1$ over the alphabet $[n+1]$ that contains every permutation of length $n$ as a factor.

\section{As Subsequences, Over $[n]$}
\label{sec-seq-n}

The universal permutation problem for subsequences over the alphabet $[n]$ pre-dates the others by 20 years. In a 1972 technical report entitled ``Selected Combinatorial Research Problems'' and edited together with Chv\'atal and Klarner, Knuth~\cite[Problem 36]{chvatal:selected-combin:} stated the following problem, which he attributed to Richard Karp:
\begin{quote}
What is the shortest string of $\{1,2,\dots,n\}$ containing all permutations on $n$ elements as subsequences? (For $n = 3$, $1213121$; for $n=4$, $123412314321$; for $n=5$, M. Newey claims the shortest has length $19$.)
\end{quote}
%Let $a(n)$ denote the length of the shortest word over the alphabet $[n]$ containing subsequences order-isomorphic to every permutation of length $n$. 
To this day, the lengths of the shortest universal words in this case are known exactly only for $1\le n\le 7$. These values were computed by Newey to be $1$, $3$, $7$, $12$, $19$, $28$, and $39$ in his 1973 technical report~\cite{newey:notes-on-a-prob:}, and he observed that this sequence is equal to $n^2-2n+4$ for $3\le n\le 7$. In fact, Newey gave a construction of universal words of this length for all $n\ge 3$, meaning $n^2-2n+4$ is an upper bound on the answer to this universal permutation problem. While Newey remarked that it is an ``obvious conjecture'' that the length of the shortest universal word in this case is $n^2-2n+4$, he also suggested a competing conjecture that would imply that the lengths grow like $n^2 - n \log_2(n)$. 

Simpler constructions of universal words of length $n^2-2n+4$ were presented in a 1974 paper of Adleman~\cite{adleman:short-permutati:}, a 1975 paper of Koutas and Hu~\cite{koutas:shortest-string:}, and a 1976 paper of Galbiati and Preparata~\cite{galbiati:on-permutation-:}. The latter two constructions were given a common generalization in the 1980 paper of Mohanty~\cite{mohanty:shortest-string:}. Interestingly, of these four papers, only Koutas and Hu were bold (or foolish) enough to conjecture that $n^2-2n+4$ is the true answer (it isn't).

After this initial flurry of activity, the problem laid dormant until the surprising 2011 work of Z\u{a}linescu~\cite{zu-alinescu:shorter-strings:}, who lowered the upper bound by $1$ for $n\ge 10$, constructing a word of length $n^2-2n+3$ that contains all permutations of length $n$ as subsequences. However, his upper bound stood for just over one year before being improved upon, for $n\ge 11$, by the following.

\begin{theorem}[Radomirovi\'c~\cite{radomirovic:a-construction-:}]
\label{thm-miller-perms}
For all $n\ge 7$, there is a word over the alphabet $[n]$ of length $\left\lceil n^2-7n/3+19/3\right\rceil$ containing subsequences equal to every permutation of length $n$.
\end{theorem}

For a lower bound on the length of a universal word in this context, we briefly present the elementary proof given by Kleitman and Kwiatkowski~\cite{kleitman:a-lower-bound-o:}. Let $w\in[n]^\ast$ be a word that contains each permutation of length $n$ as a subsequence. Choose $\pi(1)$ to be the symbol whose earliest occurrence in $w$ is as late as possible, and note that this occurrence may not appear before $w(n)$. Next, choose $\pi(2)$ to be the symbol whose earliest occurrence after $\pi(1)$ in $w$ is as late as possible, and note that this occurrence must be at least $n-1$ symbols later. Then, choose $\pi(3)$ to be the symbol whose earliest occurrence after $\pi(1)\pi(2)$ first appears as a subsequence in $w$ is as late as possible, and note that this means that $\pi(3)$ must occur at least $n-2$ symbols later. Continuing in this manner, we construct a permutation $\pi$ whose earliest possible occurrence in $w$ requires at least
\[
	n + (n-1) + \cdots + 2 + 1
	=
	\frac{n^2+n}{2}
%	{n+1 \choose 2}
\]
symbols.

Kleitman and Kwiatkowski~\cite{kleitman:a-lower-bound-o:} go on to prove (via a delicate inductive argument) a lower bound of $n^2-c_\epsilon n^{(7/4)+\epsilon}$, where the constant $c_\epsilon$ depends on $\epsilon$. While this later bound lacks concreteness, it does establish that the lengths of the shortest universal words in this case are asymptotic to $n^2$.

\section{As Subsequences, Over $[n+1]$}
\label{sec-seq-n+1}

As in the factor case, by adding a single symbol to our alphabet, we again see a dramatic decrease in the length of the shortest universal word. To date, this version of the problem has only been studied implicitly, in the 2009 work of Miller~\cite{miller:asymptotic-boun:}, where she established the following bound.

\begin{theorem}[Miller~\cite{miller:asymptotic-boun:}]
\label{thm-miller-perms}
For all $n\ge 1$, there is a word over the alphabet $[n+1]$ of length $(n^2+n)/2$ containing subsequences order-isomorphic to every permutation of length $n$.
\end{theorem}

To establish this result, define the \emph{infinite zigzag word} to be the word formed by alternating between ascending \emph{runs} of the odd positive integers $1357\cdots$ and descending \emph{runs} of the even positive integers $\cdots 8642$,
\[
	1357\cdots\ \cdots 8642\ 
	1357\cdots\ \cdots 8642\ 
	1357\cdots\ \cdots 8642\ \cdots.
\]
While this object does not conform to most definitions of the word \emph{word} in combinatorics, we hope the reader forgives us the slight expansion of the definition adopted here. We are interested in the leftmost embeddings of words over $\mathbb{P}$ into the infinite zigzag word.

We also need two definitions. First, given a word $p\in\mathbb{P}^\ast$, we define the word $p^{+1}\in\mathbb{P}^\ast$ to be the word formed by adding $1$ to each letter of $p$, so $p^{+1}(i)=p(i)+1$ for all indices $i$ of $p$. Next we say that the word $p\in\mathbb{P}^\ast$ has an \emph{immediate repetition} if there is an index $i$ with $p(i)=p(i+1)$, i.e., if $p$ contains a factor equal to $\ell\ell$ for some letter $\ell\in\mathbb{P}$.

\begin{proposition}
\label{prop-miller-words}
If the word $p\in\mathbb{P}^n$ has no immediate repetitions, then either $p$ or $p^{+1}$ occurs as a subsequence of the first $n$ runs of the infinite zigzag word.
\end{proposition}

Before proving Proposition~\ref{prop-miller-words}, note that permutations do not have immediate repetitions. Thus if $\pi$ is a permutation of length $n$, Proposition~\ref{prop-miller-words} implies that either $\pi$ or $\pi^{+1}$ occurs as a subsequence in the first $n$ runs of the infinite zigzag word. Since $\pi^{+1}$ is order-isomorphic to $\pi$ and both $\pi$ and $\pi^{+1}$ are words over $[n+1]$, this implies that the restriction of the first $n$ runs of the infinite zigzag word to the alphabet $[n+1]$ contains every permutation of length $n$. For example, in the case of $n=5$ we obtain the universal word
\[
	135\ 642\ 135\ 642\ 135
\]
of length $15$ over the alphabet $[6]$.

The restriction of the infinite zigzag word described above consists of $n$ runs of average length $(n+1)/2$: if $n$ is odd, then all runs are of this length, while if $n$ is even, then half are of length $n/2$ and half are of length $(n+2)/2$. Thus Proposition~\ref{prop-miller-words} implies Theorem~\ref{thm-miller-perms}. While Proposition~\ref{prop-miller-words} does not appear explicitly in Miller~\cite{miller:asymptotic-boun:}, its proof, presented below, is adapted from her proof of Theorem~\ref{thm-miller-perms}.

\newenvironment{proof-of-prop-miller-words}{\medskip\noindent {\it Proof of Proposition~\ref{prop-miller-words}.\/}}{\qed\bigskip}
\begin{proof-of-prop-miller-words}
We define the \emph{score} of the word $p\in\mathbb{P}^\ast$, denoted by $s(p)$, as the minimum number of runs that an initial segment of the infinite zigzag word must have in order to contain $p$, minus the length of $p$. Thus our goal is to show that for every word $p\in\mathbb{P}^\ast$ without immediate repetitions, either $s(p)\le 0$ or $s(p^{+1})\le 0$. In fact, we show that for such words we have $s(p)+s(p^{+1})=1$, which implies this.

We prove this claim by induction on the length of $p$. For the base case, we see that words consisting of a single odd letter are contained in the first run of the infinite zigzag word (thus corresponding to scores of $0$) while words consisting of a single even letter are contained in the second run (corresponding to scores of $1$). Thus for every $\ell\in\mathbb{P}^1$ we have $s(\ell)+s(\ell^{+1})=1$, as desired. Now suppose that the claim is true for all words $p\in\mathbb{P}^n$ without immediate repetitions and let $\ell\in\mathbb{P}$ denote a letter. We see that, for any $p \in \mathbb{P}^n$,
\[
	s(p\ell)-s(p)
	=
	\left\{
	\begin{array}{cl}
		-1&	\begin{array}{l}
			\text{if $p(n)<\ell$ and both entries are odd or}\\
			\text{if $p(n)>\ell$ and both entries are even;}
			\end{array}
		\\[12pt]
		0&	\begin{array}{l}
			\text{if $p(n)$ and $\ell$ are of different parity; or}
			\end{array}
		\\[8pt]
		+1&	\begin{array}{l}
			\text{if $p(n)<\ell$ and both entries are even,}\\
			\text{if $p(n)=\ell$, or}\\
			\text{if $p(n)>\ell$ and both entries are odd.}
			\end{array}
	\end{array}
	\right.
\]
Because our words do not have immediate repetitions, we can ignore the possibility that $\ell=p(n)$. In the other cases, it can be seen by inspection that
\[
	\big(  s(p\ell)-s(p)  \big)
	+
	\big(  s\!\left((p\ell)^{+1}\right) - s\!\left(p^{+1}\right)\!  \big)
	=
	0.
\]
By rearranging these terms, we see that
\[
	s(p\ell) + s\!\left((p\ell)^{+1}\right)
	=
	s(p)+s\!\left(p^{+1}\right).
\]
Since $s(p)+s\!\left(p^{+1}\right)=1$ by induction, this completes the proof of the inductive claim, and thus also of the proposition.
\end{proof-of-prop-miller-words}

We conclude our consideration of this case by providing a lower bound. Suppose that the word $w$ over the alphabet $[n+1]$ contains subsequences order-isomorphic to every permutation of length $n$. For each letter $\ell\in[n+1]$, let $r_\ell$ denote the number of occurrences of the letter $\ell$ in $w$. To create a subsequence of $w$ that is order-isomorphic to a permutation, we must choose a letter of the alphabet $[n+1]$ to omit and then choose precisely one occurrence of each of the other letters. Thus the number of permutations that can be contained in $w$ is at most
\[
	\sum_{\ell\in[n+1]} r_1\cdots r_{\ell-1}r_{\ell+1}\cdots r_{n+1}
	=
	r_1 \cdots r_{n+1} \sum_{\ell \in [n+1]} \frac{1}{r_{\ell}}.
\]
Setting $m=|w|=\sum r_\ell$, we see that the above quantity attains its maximum over all $(r_1,\dots,r_{n+1})\in\mathbb{R}_{\ge 0}^{n+1}$ when each $r_\ell$ is equal to $m/(n+1)$, and in that case the number of permutations contained in $w$ is at most
\[
	(n+1)\left(\frac{m}{n+1}\right)^n.
\]
If $w$ is to contain all permutations of length $n$, then this quantity must be at least $n!$. Using the fact that $k!\ge (k/e)^k$ for all $k$, we must therefore have
\[
	(n+1)\left(\frac{m}{n+1}\right)^n
	\ge
	n!
	\ge
	\left(\frac{n}{e}\right)^n.	
\]
It follows that, asymptotically, we must have $m\ge n^2/e$.

%\begin{figure}
%\begin{footnotesize}
%\begin{center}
%\begin{tabular}{ccc}
%\begin{tikzpicture}[scale=0.35, baseline=(current bounding box.south)]
%	\foreach \y/\val [count = \x] in {1/2, 3/7, 5/12, 6/14.5, 4/9.5, 2/4.5, 1/2, 3/7, 5/12, 6/14.5, 4/9.5, 2/4.5, 1/2, 3/7, 5/12} {
%		\absdot{(\x,\val)}
%		\node at (\x, 0) {\y};
%		}
%	\draw[darkgray, thick, line cap=round] (0.5,0.5) rectangle (15.5,15.5);
%	\foreach \x in {3,6,9,12} {
%		\draw[darkgray, thick, line cap=round] (\x+0.5, 0.5) -- (\x+0.5, 15.5);
%	}
%	\foreach \y in {3,5,8,10,13} {
%		\draw[darkgray, thick, line cap=round] (0.5, \y+0.5) -- (15.5, \y+0.5);
%	}
%\end{tikzpicture}
%&\quad\quad&
%\begin{tikzpicture}[scale=0.35, baseline=(current bounding box.south)]
%	\plotperm{1,6,12,14,9,5,3,7,11,15,10,4,2,8,13}
%	\foreach \y [count=\x] in {1,6,12,14,9,5,3,7,11,15,10,4,2,8,13} {
%		\node at (\x,0) {\y};
%	}
%	\draw[darkgray, thick, line cap=round] (0.5,0.5) rectangle (15.5,15.5);
%	\foreach \x in {3,6,9,12} {
%		\draw[darkgray, thick, line cap=round] (\x+0.5, 0.5) -- (\x+0.5, 15.5);
%	}
%	\foreach \y in {3,5,8,10,13} {
%		\draw[darkgray, thick, line cap=round] (0.5, \y+0.5) -- (15.5, \y+0.5);
%	}
%	\node at (0,0) {\phantom{1}};
%\end{tikzpicture}
%\end{tabular}
%\caption{Miller's construction figure. The permutation isn't necessary, given how we set up the exposition, but we'll refer to it later when we choose $\zeta_n$.}
%\label{fig-miller-construction}
%\end{center}
%\end{footnotesize}
%\end{figure}

%
%
%
%
%

\section{As Subsequences, Over $\mathbb{P}$}
\label{sec-subseq-P}

For the final cell of Table~\ref{tab-six-problems}, we seek a word over the positive integers $\mathbb{P}$ that contains all permutations of length $n$ as subsequences. As remarked upon in the Introduction, this is equivalent to seeking a permutation that contains all permutations of length $n$, and such a permutation is sometimes called an \emph{$n$-superpattern} (for example, by B\'ona~\cite[Chapter 5, Exercises 19--22 and Problems Plus 9--12]{bona:combinatorics-o:}). The first result about universal permutations of this type was obtained by Simion and Schmidt in 1985~\cite[Section~5]{simion:restricted-perm:}. They computed the number of $3$-universal permutations of length $m\ge 5$ to be
\[
	m!
	-6C_m
	+5\cdot 2^m
	+4{m\choose 2}
	-2F_m
	-14m
	+20.
\]
(Here $C_m$ denotes the $m$th Catalan number and $F_m$ denotes the $m$th \emph{combinatorial} Fibonacci number, so $F_0=F_1=1$ and $F_m=F_{m-1}+F_{m-2}$ for $m\ge 2$.) However, the first to study this version of the universal permutation problem for general $n$ was Arratia~\cite{arratia:on-the-stanley-:} in 1999. 

As our alphabet has only expanded from the version of the problem discussed in the previous section, the upper bound of $(n^2+n)/2$ established in Theorem~\ref{thm-miller-perms} also holds for the version of the problem discussed in this section. It should be noted that before Miller~\cite{miller:asymptotic-boun:} established Theorem~\ref{thm-miller-perms} in 2009, Eriksson, Eriksson, Linusson, and W{{\"{a}}}stlund~\cite{eriksson:dense-packing-o:} had established an upper bound for this problem asymptotically equal to $2n^2/3$.

Here, we give a new improvement to Miller's upper bound. In order to do so, we further restrict the infinite zigzag word, and then break ties between its letters to obtain a specific permutation $\zeta_n$. To this end, we define the word $z_n$ to be the restriction of the first $n$ runs of the infinite zigzag word to the alphabet $[n]$. When $n$ is even, each run of $z_n$ has length $n/2$. When $n$ is odd, $z_n$ consists of $(n+1)/2$ ascending odd runs, each of length $(n+1)/2$, and $(n-1)/2$ descending even runs, each of length $(n-1)/2$. Thus we have
\[
	|z_n|
	=
	\left\{
	\begin{array}{cl}
		\displaystyle\frac{n^2}{2}&\text{if $n$ is even,}\\[12pt]
		\displaystyle\frac{n^2+1}{2}&\text{if $n$ is odd.}
	\end{array}
	\right.
\]

Next we choose a specific permutation, $\zeta_n$, such that $z_n$ is order-homomorphic to $\zeta_n$. Recall that this means that for all indices $i$ and $j$,
\[
	z_n(i)>z_n(j)
	\implies
	\zeta_n(i)>\zeta_n(j).
\]
In constructing $\zeta_n$, we have the freedom to break ties between equal letters of $z_n$. That is to say, if $z_n(i)=z_n(j)$ for $i\neq j$, then in constructing $\zeta_n$ we may choose whether $\zeta_n(i)<\zeta_n(j)$ or $\zeta_n(i)>\zeta_n(j)$ arbitrarily without affecting any other pair of comparisons and thus without losing any occurrences of permutations. We choose to break these ties by replacing all instances of a given letter $k\in[n]$ in $z_n$ by a decreasing subsequence in $\zeta_n$. Thus for indices $i<j$, we have
\[
	z_n(i)=z_n(j)
	\implies
	\zeta_n(i)>\zeta_n(j).
\]
This choice uniquely determines $\zeta_n$ (up to order-isomorphism), as all comparisons between its letters are determined either in $z_n$, if the corresponding letters of $z_n$ differ, or by the rule above, if the corresponding letters of $z_n$ are the same. Figure~\ref{fig-z5-zeta5} shows the plots of $z_5$ and $\zeta_5$, where the \emph{plot} of a word $w$ over $\mathbb{P}$ is the set $\{(i,w(i))\}$ of points in the plane.

\begin{figure}
\begin{footnotesize}
\begin{center}
\begin{tabular}{ccc}
\begin{tikzpicture}[scale=0.35, baseline=(current bounding box.south)]
	\foreach \y/\val in {1/2,2/4.5,3/7,4/9.5,5/12} {
%		\node at (0,\val) {\y};
%		\draw [thin] (0.5,\val) -- (13.5,\val);	
		}
	\foreach \y/\val [count = \x] in {1/2, 3/7, 5/12, 4/9.5, 2/4.5, 1/2, 3/7, 5/12, 4/9.5, 2/4.5, 1/2, 3/7, 5/12} {
		\absdot{(\x,\val)}
		\node at (\x, 0) {\y};
		}
	\draw[darkgray, thick, line cap=round] (0.5,0.5) rectangle (13.5,13.5);
	\foreach \y in {3,5,8,10} {
		\draw[darkgray, thick, line cap=round] (0.5, \y+0.5) -- (13.5, \y+0.5);
		\draw[darkgray, thick, line cap=round] (\y+0.5, 0.5) -- (\y+0.5, 13.5);
	}
\end{tikzpicture}
&\quad\quad&
\begin{tikzpicture}[scale=0.35, baseline=(current bounding box.south)]
	\plotperm{3,8,13,10,5,2,7,12,9,4,1,6,11}
%	\foreach \y [count=\x] in {3,8,13,10,5,2,7,12,9,4,1,6,11} {
%		\node at (\x,0) {\y};
%%		\node at (0,\x) {\x};
%	}
	\node at (1,0) {3};
	\node at (2,0) {8};
	\node at (3,0) {1\!\!\:3};
	\node at (4,0) {1\!\!\:0};
	\node at (5,0) {5};
	\node at (6,0) {2};
	\node at (7,0) {7};
	\node at (8,0) {1\!\!\:2};
	\node at (9,0) {9};
	\node at (10,0) {4};
	\node at (11,0) {1};
	\node at (12,0) {6};
	\node at (13,0) {1\!\!\:1};
	\draw[darkgray, thick, line cap=round] (0.5,0.5) rectangle (13.5,13.5);
	\foreach \y in {3,5,8,10} {
		\draw[darkgray, thick, line cap=round] (0.5, \y+0.5) -- (13.5, \y+0.5);
		\draw[darkgray, thick, line cap=round] (\y+0.5, 0.5) -- (\y+0.5, 13.5);
	}
	\node at (0,0) {\phantom{1}};
\end{tikzpicture}
\end{tabular}
\caption{On the left, the further restriction we define of the infinite zigzag word, $z_5$. On the right, the normalized permutation formed by breaking ties, $\zeta_5$.}
\label{fig-z5-zeta5}
\end{center}
\end{footnotesize}
\end{figure}

In the following sequence of results, we show that $\zeta_n$ is \emph{almost} universal. In fact, we show that $\zeta_n$ fails to be universal only for even $n$, and in that case, the only missing permutation is the decreasing permutation $n\cdots 21$. The first of these results, Proposition~\ref{prop-distant-inv-desc}, covers almost all permutations. (In fact, Proposition~\ref{prop-distant-inv-desc-layered} shows that Proposition~\ref{prop-distant-inv-desc} handles all but $2^{n-1}$ permutations of length $n$.)

We say that two entries $\pi(j)$ and $\pi(k)$ form an \emph{inverse-descent} if $j<k$ and $\pi(j)=\pi(k)+1$. (As the name is meant to indicate, if a pair of entries forms an inverse-descent in $\pi$, then the corresponding entries of $\pi^{-1}$ form a descent.) If $\pi(j)$ and $\pi(k)$ form an inverse-descent and they are not adjacent in $\pi$ (so $k\ge j+2$), then we say that they form a \emph{distant} inverse-descent.

\begin{proposition}
\label{prop-distant-inv-desc}
If the permutation $\pi$ of length $n$ has a distant inverse-descent, then $\zeta_n$ contains a subsequence order-isomorphic to $\pi$.
\end{proposition}
\begin{proof}
Suppose that the entries $\pi(a)$ and $\pi(b)$ form a distant inverse-descent in $\pi$, meaning that $\pi(a)=\pi(b)+1$ and $b\ge a+2$. We define the word $p\in [n-1]^n$ by
\[
	p(i)
	=
	\left\{\begin{array}{ll}
		\pi(i)&\text{if $\pi(i)\le\pi(b)$,}\\
		\pi(i)-1&\text{if $\pi(i)\ge\pi(a)=\pi(b)+1$.}
	\end{array}\right.
\]

The word $p$ has two occurrences of the letter $\pi(b)$, but because $\pi(a)$ and $\pi(b)$ form a distant inverse-descent, these two occurrences of $\pi(b)$ in $p$ do not constitute an immediate repetition. Thus Proposition~\ref{prop-miller-words} shows that either $p$ or $p^{+1}$ occurs as a subsequence in the first $n$ runs of the infinite zigzag word. As $p$ and $p^{+1}$ are both words over $[n]$, whichever of these words occurs in the first $n$ runs of the infinite zigzag word  also occurs as a subsequence of $z_n$. Suppose that this subsequence occurs in the indices $1\le i_1<i_2<\cdots<i_n\le |z_n|$, so $z_n(i_1)z_n(i_2)\cdots z_n(i_n)$ is equal to either $p$ or $p^{+1}$, and thus for $j,k \in [n]$ we have
\[
	z_n(i_j) > z_n(i_k)
	\iff 
	p(j) > p(k).
\]
Because $z_n$ is order-homomorphic to $\zeta_n$, this implies that for all pairs of indices $j,k\in[n]$ except the pair $\{a,b\}$, we have
\[
	\zeta_n(i_j) > \zeta_n(i_k)
	\iff 
	p(j) > p(k)
	\iff 
	\pi(j) > \pi(k).
\]
Furthermore, since $p(a)=p(b)$, we have $z_n(a)=z_n(b)$, and so by our construction of $\zeta_n$ it follows that $\zeta_n(a)>\zeta_n(b)$, while we know that $\pi(a)>\pi(b)$ because those entries form an inverse-descent. This verifies that $\zeta_n(i_1)\zeta_n(i_2)\cdots \zeta_n(i_n)$ is order-isomorphic to $\pi$, completing the proof.
\end{proof}

\begin{figure}
	\[
	\pi\oplus\sigma
	=\,\,
	\begin{tikzpicture}[scale=0.2, baseline=(current bounding box.center)]
		\plotpermbox{1}{1}{4}{4};
		\plotpermbox{5}{5}{8}{8};
		\node at (2.5,2.5) {$\pi$};
		\node at (6.5,6.5) {$\sigma$};
	\end{tikzpicture}
	\quad\quad\quad\quad
	\begin{tikzpicture}[scale=0.2, baseline=(current bounding box.center)]
		% The permutation:
		\plotpermbox{1}{1}{2}{2};
		\plotpermbox{3}{3}{3}{3};
		\plotpermbox{4}{4}{6}{6};
		\plotpermbox{7}{7}{8}{8};
		\begin{scope}[shift={(0,-1pt)}]
			\plotperm{2,1,3,6,5,4,8,7};
		\end{scope}
	\end{tikzpicture}
	\]
\caption{The plot of the sum of $\pi$ and $\sigma$ is shown on the left. The figure on the right shows the plot of the layered permutation $21\ 3\ 654\ 87$ with layer lengths $2$, $1$, $3$, $2$.}
\label{fig-sum-and-layered}
\end{figure}
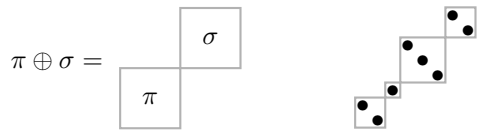

To describe the permutations that Proposition~\ref{prop-distant-inv-desc} does not apply to, we need the notions of sums of permutations and layered permutations. Given permutations $\pi$ and $\sigma$ of respective lengths $m$ and $n$, their \emph{(direct) sum} is the permutation $\pi\oplus\sigma$ of length $m+n$ defined by
\[
	(\pi\oplus\sigma)(i)
	=
	\left\{\begin{array}{ll}
		\pi(i)&\text{if $1\le i\le m$,}\\
		\sigma(j-m)+m&\text{if $m+1\le i\le m+n$.}
	\end{array}\right.
\]
Pictorially, the plot of $\pi\oplus\sigma$ then consists of the plot of $\sigma$ placed above and to the right of the plot of $\pi$, as shown on the left of Figure~\ref{fig-sum-and-layered}. A permutation is said to be \emph{layered} if it can be expressed as a sum of decreasing permutations, and in this case, these decreasing permutations are themselves called the \emph{layers}. An example of a layered permutation is shown on the right of Figure~\ref{fig-sum-and-layered}.

\begin{proposition}
\label{prop-distant-inv-desc-layered}
The permutation $\pi$ is layered if and only if it does not have a distant inverse-descent.
\end{proposition}
\begin{proof}
One direction is completely trivial: if $\pi$ is layered then all of its inverse-descents are between consecutive entries, so it does not have a distant inverse-descent. For the other direction we use induction on the length of $\pi$. The empty permutation is layered, so the base case holds. If $\pi$ is a nonempty permutation without distant inverse-descents, then it must begin with the entries $\pi(1)$, $\pi(1)-1$, $\dots$, $2$, $1$ in that order. This means that $\pi=\delta\oplus\sigma$ where $\delta$ is a nonempty decreasing permutation and $\sigma$ is a permutation shorter than $\pi$ that also does not have any distant inverse-descents. By induction, $\sigma$ is layered, and thus $\pi$ is as well, completing the proof.
\end{proof}

Having characterized the permutations to which Proposition~\ref{prop-distant-inv-desc} does not apply, we now show that almost all of them are nevertheless contained in $\zeta_n$.

\begin{proposition}
\label{prop-layered-zeta}
If the permutation $\pi$ of length $n$ is layered and not a decreasing permutation of even length, then $\zeta_n$ contains a subsequence order-isomorphic to $\pi$.
\end{proposition}
\begin{proof}
Let $\pi$ denote an arbitrary layered permutation of length $n$. To prove the result, we compute the score of $\pi$ as in the proof of Proposition~\ref{prop-miller-words}, show that this score can only take on the values $0$ or $\pm 1$, and then describe an alternative embedding of $\pi$ in $\zeta_n$ in the case where the score of $\pi$ is $1$, except when $\pi$ is a decreasing permutation of even length.

Recall that the score of any word $\pi$, $s(\pi)$, is defined as the number of initial runs of the infinite zigzag word necessary to contain $\pi$ minus the length of $\pi$. As observed in the proof of Proposition~\ref{prop-miller-words}, the score of a word does not change upon reading a letter of opposite parity. This implies that, while reading a layered permutation, the score changes only when transitioning from one layer to the next, and thus we compute the score of $\pi$ layer-by-layer.

\renewcommand{\oe}{\small\textsf{odd}\text{--}\textsf{even}}
\newcommand{\oo}{\small\textsf{odd}\text{--}\textsf{odd}}
\newcommand{\ee}{\small\textsf{even}\text{--}\textsf{even}}
\newcommand{\eo}{\small\textsf{even}\text{--}\textsf{odd}}

\begin{figure}
\begin{footnotesize}
\begin{center}
%	\begin{tikzpicture}[scale=1, xscale=3, shorten >= 5pt, shorten <= 5pt]
	\begin{tikzpicture}[
		scale=1, 
		xscale=3, 
		node style/.style={thick, draw, ellipse, minimum width=68pt, minimum height = 16.66666pt, align=center}
	]

		\draw (-1, 0) node[node style] (oe) {$\oe$};
		\draw ( 0,-1) node[node style] (ee) {$\ee$};
		\draw ( 0, 1) node[node style] (oo) {$\oo$};
		\draw ( 1, 0) node[node style] (eo) {$\eo$};
		
		\draw [->] (0,-{1.6}) to (0,-{1.333333});
		
%		\draw [->] (eo) to (oo);
		\draw [-] (+0.9, +{1/3}) to (+0.9,+0.8);
		\draw [domain=0:90] plot ({+0.833333+0.066667*cos(\x)}, {+0.8+0.2*sin(\x)});
		\draw [->] (+0.833333,+1.0) to (+0.425, +1.0);
		\node at (0.8,0.8) {$-1$};
		
%		\draw [->] (oo) to (oe);
		\draw [-] (-0.425, 1.0) to (-0.833333,1.0);
		\draw [domain=90:180] plot ({-0.833333+0.066667*cos(\x)}, {0.8+0.2*sin(\x)});
		\draw [->] (-0.9,0.8) to (-0.9, {1/3});
		\node at (-0.8,0.8) {$-1$};
		
%		\draw [->] (oe) to (ee);
		\draw [-] (-0.9, -{1/3}) to (-0.9,-0.8);
		\draw [domain=180:270] plot ({-0.833333+0.066667*cos(\x)}, {-0.8+0.2*sin(\x)});
		\draw [->] (-0.833333,-1.0) to (-0.425, -1.0);
		\node at (-0.8,-0.8) {$+1$};
		
%		\draw [->] (ee) to (eo);
		\draw [-] (+0.425, -1.0) to (+0.833333,-1.0);
		\draw [domain=270:360] plot ({+0.833333+0.066667*cos(\x)}, {-0.8+0.2*sin(\x)});
		\draw [->] (+0.9,-0.8) to (+0.9, -{1/3});
		\node at (+0.8,-0.8) {$+1$};
		
%		\draw [->] (oe) to (oe);
		\newcommand\hmargin{0.0666667}
		\newcommand\vmargin{0.2}
		\newcommand\leftloopleft{-1.45}
		\newcommand\leftloopright{-1.166667}
		\newcommand\loopupper{0.45}
		\newcommand\looplower{-\loopupper}
		
		\draw [-] (\leftloopright+\hmargin, {1/3}) to (\leftloopright+\hmargin,\loopupper);
		\draw [domain=  0: 90] plot ({\leftloopright+\hmargin*cos(\x)}, {\loopupper+\vmargin*sin(\x)}); % NE corner
		\draw [-] (\leftloopright, \loopupper+\vmargin) to (\leftloopleft,\loopupper+\vmargin);
		\draw [domain= 90:180] plot ({\leftloopleft+\hmargin*cos(\x)}, {\loopupper+\vmargin*sin(\x)}); % NW corner
		\draw [-] (\leftloopleft-\hmargin,\loopupper) to (\leftloopleft-\hmargin, \looplower);
		\draw [domain=180:270] plot ({\leftloopleft+\hmargin*cos(\x)}, {\looplower+\vmargin*sin(\x)}); % SW corner
		\draw [-] (\leftloopright, \looplower-\vmargin) to (\leftloopleft,\looplower-\vmargin);
		\draw [domain=270:360] plot ({\leftloopright+\hmargin*cos(\x)}, {\looplower+\vmargin*sin(\x)}); % SE corner
		\draw [->] (\leftloopright+\hmargin,\looplower) to (\leftloopright+\hmargin, {-1/3});
		\node at (\leftloopleft,\loopupper) {$0$};
		
%		\draw [->] (eo) to (eo);
		\newcommand\rightloopright{-\leftloopleft}
		\newcommand\rightloopleft{-\leftloopright}

		\draw [-] (\rightloopleft-\hmargin,{-1/3}) to (\rightloopleft-\hmargin, \looplower);
		\draw [domain=180:270] plot ({\rightloopleft+\hmargin*cos(\x)}, {\looplower+\vmargin*sin(\x)}); % NE corner
		\draw [-] (\rightloopright, \looplower-\vmargin) to (\rightloopleft,\looplower-\vmargin);
		\draw [domain=270:360] plot ({\rightloopright+\hmargin*cos(\x)}, {\looplower+\vmargin*sin(\x)}); % NW corner
		\draw [-] (\rightloopright+\hmargin,\looplower) to (\rightloopright+\hmargin, \loopupper);
		\draw [domain=  0: 90] plot ({\rightloopright+\hmargin*cos(\x)}, {\loopupper+\vmargin*sin(\x)}); % SW corner
		\draw [-] (\rightloopright, \loopupper+\vmargin) to (\rightloopleft,\loopupper+\vmargin);
		\draw [domain= 90:180] plot ({\rightloopleft+\hmargin*cos(\x)}, {\loopupper+\vmargin*sin(\x)}); % SE corner
		\draw [->] (\rightloopleft-\hmargin, \loopupper) to (\rightloopleft-\hmargin,{1/3});
		\node at (\rightloopright,\looplower) {$0$};

		\node at (-0.8,-0.8) {$+1$};

%		\draw [->] (oo) to (ee);
		\draw [->] (-0.1,{2/3}) to (-0.1,{-2/3});
		\node at (-0.15,0) {$0$};
		
%		\draw [->] (ee) to (oo);
		\draw [->] (+0.1,{-2/3}) to (+0.1,{2/3});
		\node at (+0.15,0) {$0$};
		
	\end{tikzpicture}
\end{center}
\end{footnotesize}
\caption{A directed graph describing the scoring of a layered permutation.}
\label{fig-layered-zeta-automaton}
\end{figure}
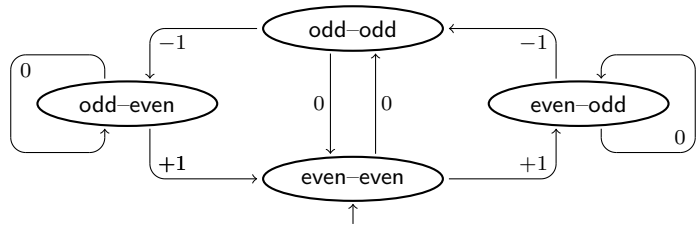

The change in score when moving from one layer of $\pi$ to the next is determined by the parity of the last entry of the layer we are leaving and the first entry of the layer we are entering. Specifically, the score changes by $-1$ if both of these entries are odd and $+1$ if both are even. This shows that in order to compute the score of the layered permutation $\pi$, we simply need to know the parities of the first and last entries of each of its layers. This information is represented by the labels of the nodes of the directed graph shown in Figure~\ref{fig-layered-zeta-automaton}.

Moreover, not all transitions between these nodes are possible, because the last entry of a layer is precisely $1$ greater than the first entry of the preceding layer. This is why there are only eight edges shown in Figure~\ref{fig-layered-zeta-automaton}. In this figure, each of those edges is labeled by the change in the score function. Note that the first layer must end with $1$ (an odd entry), and its first entry must be either odd (for a score of $0$) or even (for a score of $1$); this is equivalent to starting our walk on the graph in Figure~\ref{fig-layered-zeta-automaton} at the node labeled $(\ee)$ before any layers are read.

From this graphical interpretation of the scoring process, it is apparent that the score of a layered permutation can take on only three values: $-1$ if it ends at the node $(\oe)$; $0$ if it ends at either node $(\ee)$ or $(\oo)$; or $1$ if it ends at the node $(\eo)$. Except in this final case, we are done.

Now suppose that we are in the final case, so the ultimate layer of $\pi$ is of $(\eo)$ type. The first entry of this layer is the greatest entry of $\pi$, so we know that $\pi$ has even length. If $\pi$ were a decreasing permutation then there would be nothing to prove (as we have not claimed anything in this case), so let us further suppose that $\pi$ is not a decreasing permutation, and thus that $\pi$ has at least two layers. We further divide this case into two cases. In both cases, as in the proof of Proposition~\ref{prop-distant-inv-desc}, we construct a word $p\in[n-1]^n$ such that if $z_n$ contains $p$, then $\zeta_n$ contains $\pi$.

First, suppose that the penultimate layer of $\pi$ is of $(\eo)$ type and that this layer begins with the entry $\pi(b)$. This implies that the penultimate layer of $\pi$ has at least two entries (because its first and last entries have different parities). In this case, we define $p$ by
\[
	p(i)
	=
	\left\{\begin{array}{ll}
		\pi(i)&\text{if $\pi(i)<\pi(b)$,}\\
		\pi(i)-1&\text{if $\pi(i)\ge\pi(b)$.}
	\end{array}\right.
\]
In other words, to form $p$ from $\pi$ we decrement the first entry of the penultimate layer and all entries of the ultimate layer. Because the penultimate layer of $\pi$ has at least two entries, performing this operation creates an immediate repetition (of the entry $\pi(b)-1$) at the beginning of this layer. For example, if $\pi=21\ 6543\ 87$ then $\pi(b)=6$ and we decrement the $6$, $8$, and $7$ to obtain the word $p=21\ 5543\ 76$.

As with our previous constructions, if $z_n$ contains an occurrence of $p$, then $\zeta_n$ will contain a copy of $\pi$. We  establish that $z_n$ contains $p$ by showing that $s(p)=0$, which requires a further bifurcation into subcases. In both subcases, the scoring of $p$ is computed by considering its score in the antepenultimate layer (the layer immediately before the penultimate layer), the score change when reading the newly decremented first entry of the penultimate layer, the score penalty of $+1$ because $p$ contains an immediate repetition (namely, $\pi(b)-1$ occurs twice in a row), and finally the score change between the penultimate and ultimate layers. We label these cases by the final three nodes of the directed graph from Figure~\ref{fig-layered-zeta-automaton} visited while computing the score of $\pi$.
\begin{itemize}
\item The final three layers are of type $(\ee)(\eo)(\eo)$. Note that this case includes the possibility that $\pi$ has only two layers. If $p$ has an antepenultimate layer, then the score while reading that layer is $0$ and the ascent between its last entry and the newly decremented first entry of the penultimate layer is of different parity (even to odd), contributing $0$ to the score. If $p$ does not have an antepenultimate layer, then $p$ begins with the newly decremented first entry of its penultimate layer, which contributes $0$ to the score. In either case, the score of $p$ is $0$ upon reading the first entry of the penultimate layer. The immediate repetition in the penultimate layer contributes $+1$ to the score, while the ascent between the last entry of the penultimate layer and the newly decremented first entry of the ultimate layer is odd and thus contributes $-1$, so $s(p) = 0$.
\item The final three layers are of type $(\eo)(\eo)(\eo)$. The score while reading the antepenultimate layer is $+1$. The ascent between the last entry of the antepenultimate layer and the newly decremented first entry of the penultimate layer is odd, so it contributes $-1$ to the score, the immediate repetition in the penultimate layer contributes $+1$, and the ascent between the last entry of the penultimate layer and the newly decremented first entry of the ultimate layer is odd and thus contributes $-1$, so $s(p) = 0$.
\end{itemize}

It remains to treat the case where the penultimate layer is of $(\ee)$ type. Note that this case includes the possibility that the penultimate layer consists of a single entry. Suppose that the penultimate layer ends with the entry $\pi(a)$. We define $p$ by
\[
	p(i)
	=
	\left\{\begin{array}{ll}
		\pi(i)&\text{if $\pi(i)<\pi(a)$ or $\pi(i)=n$,}\\
		\pi(i)+1&\text{if $\pi(i)\ge\pi(a)$ and $\pi(i)\neq n$.}
	\end{array}\right.
\]
Thus in forming $p$ from $\pi$ we increment all entries of the penultimate layer and all but the first entry of the ultimate layer. For example, if $\pi=21\ 3\ 654\ 87$, then we increment the $6$, $5$, $4$, and $7$ to obtain the word $p=21\ 3\ 765\ 88$.

As before, if $z_n$ contains an occurrence of $p$ then $\zeta_n$ will contain a copy of $\pi$. Thus we need only show that $s(p)=0$, which we do, as in the previous case, by considering the scoring of the final three layers. As in that case, we identify two subcases.
\begin{itemize}
\item The final three layers are of type $(\oe)(\ee)(\eo)$. The score while reading the antepenultimate layer is $-1$. The ascent between the last entry of the antepenultimate layer and the newly incremented first entry of the penultimate layer is of different parity (even to odd) and thus contributes $0$ to the score. The ascent between the newly incremented last entry of the penultimate and the first entry of the ultimate layer (which is $n$) is of different parity (odd to even) and thus contributes $0$ to the score. Finally, the immediate repetition at the beginning of the ultimate layer (the two entries equal to $n$) contributes $+1$ to the score, so $s(p)=0$.
\item The final three layers are of type $(\oo)(\ee)(\eo)$. The score while reading the antepenultimate layer is $0$. The ascent between the last entry of the antepenultimate layer and the newly incremented first entry of the penultimate layer contributes $-1$ to the score (as both entries are now odd). The ascent between the newly incremented last entry of the penultimate layer and the first entry of the ultimate layer (which is $n$) is of different parity (odd to even) and thus contributes $0$ to the score. Finally, the immediate repetition at the beginning of the ultimate layer contributes $+1$ to the score, so $s(p)=0$.
\end{itemize}

As we have considered all of the cases, the proof is complete.
\end{proof}

It remains only to conclude. The length of $\zeta_n$ is $(n^2+1)/2$ when $n$ is odd and $n^2/2$ when $n$ is even. When $n$ is odd, we have established that $\zeta_n$ is universal. However, Proposition~\ref{prop-layered-zeta} shows that $\zeta_n$ need not be universal when $n$ is even. (Indeed, it can be checked that $\zeta_n$ is \emph{not} universal when $n$ is even.) However, in this case we know that $\zeta_n$ contains the decreasing permutation $(n-1)\cdots 21$ (for instance because it contains the permutation $(n-1)\cdots 21\oplus 1$). Thus we obtain a universal permutation by prepending a new maximum entry to $\zeta_n$, giving us the following bound.

\begin{theorem}
\label{thm-perms-perms}
There is a word over $\mathbb{P}$ of length $\left\lceil(n^2+1)/2\right\rceil$ containing subsequences order-isomorphic to every permutation of length $n$.
\end{theorem}

A computer search reveals that the bound in Theorem~\ref{thm-perms-perms} is best possible for $n\le 5$. Alas, for $n=6$ the bound in the Theorem~\ref{thm-perms-perms} is $19$, but Arnar Arnarson [private communication] has found that the permutation
\[
	6\ 14\ 10\ 2\ 13\ 17\ 5\ 8\ 3\ 12\ 9\ 16\ 1\ 7\ 11\ 4\ 15
	% [6, 14, 10, 2, 13, 17, 5, 8, 3, 12, 9, 16, 1, 7, 11, 4, 15]
\]
of length $17$ is universal for the permutations of length $6$. Computations have shown that no shorter permutation is universal for the permutations of length $6$.

The best lower bound in this case is still the one given by Arratia~\cite{arratia:on-the-stanley-:} in his initial work on the problem. Note that if the word $w$ of length $m$ over the alphabet $\mathbb{P}$ is to contain subsequences order-isomorphic to each permutation of length $n$, then we must have
\[
	{m\choose n}\ge n!.
\]
As in the analysis of the lower bound of the previous section, using the fact that $k!\ge (k/e)^k$ for all $k$, we see that for the above inequality to hold we must have
\[
	\bigg(\frac{me}{n}\bigg)^n
	\ge
	{m\choose n}
	\ge
	n!
	\ge
	\bigg(\frac{n}{e}\bigg)^n,
\]
from which it follows that we must have $m\ge n^2/e^2$. In fact, Arratia~\cite[Conjecture 2]{arratia:on-the-stanley-:} conjectured that the length of the shortest universal permutation in this case is asymptotic to $n^2/e^2$.

\section{Further Variations}

In case the infinitely many problems introduced so far are not enough, we conclude by briefly describing further variants that have received attention.
\begin{enumerate}

\item As observed in Section~\ref{sec-factors-n}, there is no universal cycle over the alphabet $[n]$ for the permutations of length $n$. However, Jackson~\cite{jackson:universal-cycle:} proved that there is a universal cycle over the alphabet $[n]$ for all shorthand encodings of permutations of length $n$, where the \emph{shorthand encoding} of the permutation $\pi$ of length $n$ is the word $\pi(1)\cdots\pi(n-1)$. This result and some extensions are discussed in \cite[Section 7.2.1.2, Exercises 111--113]{knuth:the-art-of-comp:4a}, where Knuth asked for an explicit construction of such a universal cycle (Jackson's proof was nonconstructive). Knuth's request was answered by Ruskey and Williams~\cite{ruskey:an-explicit-uni:}. Further constructions have been given by Holroyd, Ruskey, and Williams~\cite{holroyd:faster-generati:,holroyd:shorthand-unive:}.

\begin{figure}
\begin{footnotesize}
\begin{center}
	\begin{tikzpicture}[scale=0.375, baseline=(current bounding box.south)]
		\foreach \val [count = \i] in {1,2,3,4,5,2,6,1,5,4,3,1,2,5,1,6,2,5} {
			\draw ({90-(\i-1)*360/18}: 4) node{\val};
		}
		\draw [->] (0,1) arc (90:-180: 1);
	\end{tikzpicture}
\quad\quad\quad\quad
	\begin{tikzpicture}[scale=0.375, baseline=(current bounding box.south)]
		\foreach \val [count = \i] in {1,2,3,4,5,6,3,1,2,6,4,3,6,5} {
			\draw ({90-(\i-1)*360/14}: 4) node{\val};
		}
		\draw [<->] (0,1) arc (90:-180: 1);
		%\draw [->] (-0.6,0) arc (-180:90: 0.6);
	\end{tikzpicture}
\vspace{0.15in}
\caption{Two rosaries presented by Gupta for permutations of length $6$. The rosary on the left may only be read clockwise, while the rosary on the right may be read either clockwise or counterclockwise.}
\label{fig-rosary}
\end{center}
\end{footnotesize}
\end{figure}
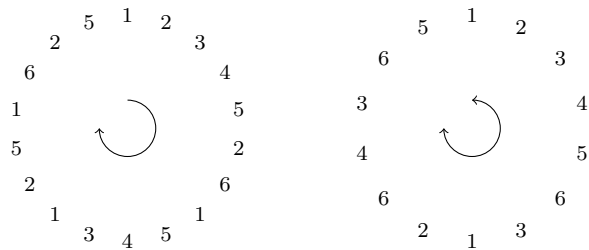

\item Gupta~\cite{gupta:on-permutation-:} considered a subsequence version of a universal cycle for permutations. A \emph{rosary} is a word $w$ over the alphabet $[n]$ such that every permutation of length $n$ is contained as a subsequence of the word
	\[
		w(k)w(k+1)\cdots w(|w|)w(1)w(2)\cdots w(k-1)
	\]
	for some value of $k$. In other words, thinking of the letters as being arranged in a circle as on the left of Figure~\ref{fig-rosary}, we may start anywhere we like, but must traverse the rosary clockwise, and cannot return to where we started. Gupta conjectured that one could always construct a rosary of length at most $n^2/2$. This conjecture was discussed by Guy~\cite[Problem E22]{guy:unsolved-proble:book} and proved in the case where $n$ is even by Lecouturier and Zmiaikou~\cite{lecouturier:on-a-conjecture:}. Gupta also considered the variant where one is allowed to traverse the rosary both clockwise and counterclockwise (see the right of Figure~\ref{fig-rosary}); he conjectured that one can always construct a rosary of length at most $3n^2/8+1/2$ in this version of the problem.

\item Albert and West~\cite{albert:universal-cycle:} studied the existence of universal cycles in the sense of Section~\ref{sec-factors-n+1} for permutation classes, making no restrictions on the size of the alphabet. To describe their results, we define a partial order on the set of all finite permutations where $\sigma\le\pi$ if $\pi$ contains a subsequence that is order-isomorphic to $\sigma$. If $\sigma\not\le\pi$ then we say that $\pi$ \emph{avoids} $\sigma$. A \emph{permutation class} is a set closed downward in this order. Every permutation class can be specified by giving the set of minimal elements \emph{not} in the class (this set is called the \emph{basis} of the class), and when presented in this form, we use the notation
\[
	\Av(B)=\{\pi\st\pi\mbox{ avoids all $\beta\in B$}\}.
\]
Most of Albert and West's results are negative in nature, but some classes they consider, such as $\Av(132,312)$, do have universal cycles over the alphabet $\mathbb{P}$. They say that a permutation class with such a universal cycle is \emph{value cyclic}.

\begin{figure}
{\footnotesize % sets font size 
\begin{center}
\begin{tikzpicture}[x=.253cm,y=3cm]
	% Axes
	\draw (0,0) -- coordinate (x axis mid) (40,0);
	\draw (0,0) -- coordinate (y axis mid) (0,1.1);
	
	% Ticks
	\foreach \x in {0,5,...,40}
 		\draw (\x,1pt) -- (\x,-3pt)
		node[anchor=north] {\x};
		
	\foreach \x in {5,9,13,17}
 		\draw (\x,1pt) -- (\x,-3pt)
		node[anchor=north] {};
	
	\draw (1pt,0) -- (-3pt,0) node[anchor=east] {$0$};
	\draw (1pt,0.5) -- (-3pt,0.5) node[anchor=east] {$\displaystyle\frac{1}{2}$};
	\draw (1pt,1) -- (-3pt,1) node[anchor=east] {$1$};

	% Labels      
%	\node [below=0.8cm] at (x axis mid) {N};
%	\node [rotate=90, above=0.8cm] at (y axis mid) {fraction};
	
	% Data
	% What fraction of n-permutations are 3-universal?
%	\draw [thick] 
%		plot [mark=none] 
%		file {data/Universal-Fractions-3-10000.txt};
	\draw [thick] plot [mark=none] coordinates
		{(0,0.000000000) (1,0.000000000) (2,0.000000000) (3,0.000000000) (4,0.000000000) (5,0.017500000) (6,0.305200000) (7,0.597900000) (8,0.818500000) (9,0.923600000) (10,0.973600000) (11,0.992300000) (12,0.997500000) (13,0.999000000) (14,0.999600000) (15,1.000000000) (16,1.000000000) (17,1.000000000) (18,1.000000000) (19,1.000000000) (20,1.000000000) (21,1.000000000) (22,1.000000000) (23,1.000000000) (24,1.000000000) (25,1.000000000) (26,1.000000000) (27,1.000000000) (28,1.000000000) (29,1.000000000) (30,1.000000000) (31,1.000000000) (32,1.000000000) (33,1.000000000) (34,1.000000000) (35,1.000000000) (36,1.000000000) (37,1.000000000) (38,1.000000000) (39,1.000000000) (40,1.000000000)};
	% What fraction of n-permutations are 4-universal?
%	\draw [thick] 
%		plot [mark=none] 
%		file {data/Universal-Fractions-4-10000.txt};
	\draw [thick] plot [mark=none] coordinates
		{(0,0.000000000) (1,0.000000000) (2,0.000000000) (3,0.000000000) (4,0.000000000) (5,0.000000000) (6,0.000000000) (7,0.000000000) (8,0.000000000) (9,0.010500000) (10,0.088100000) (11,0.265800000) (12,0.478900000) (13,0.667800000) (14,0.813700000) (15,0.899300000) (16,0.952600000) (17,0.980700000) (18,0.991200000) (19,0.996000000) (20,0.999100000) (21,0.999300000) (22,0.999500000) (23,0.999800000) (24,1.000000000) (25,1.000000000) (26,1.000000000) (27,1.000000000) (28,1.000000000) (29,1.000000000) (30,1.000000000) (31,1.000000000) (32,1.000000000) (33,1.000000000) (34,1.000000000) (35,1.000000000) (36,1.000000000) (37,1.000000000) (38,1.000000000) (39,1.000000000) (40,1.000000000)};
	% What fraction of n-permutations are 5-universal?
%	\draw [thick] 
%		plot [mark=none] 
%		file {data/Universal-Fractions-5-10000.txt};
	\draw [thick] plot [mark=none] coordinates
		{(0,0.000000000) (1,0.000000000) (2,0.000000000) (3,0.000000000) (4,0.000000000) (5,0.000000000) (6,0.000000000) (7,0.000000000) (8,0.000000000) (9,0.000000000) (10,0.000000000) (11,0.000000000) (12,0.000000000) (13,0.000000000) (14,0.001400000) (15,0.016800000) (16,0.065200000) (17,0.177700000) (18,0.320400000) (19,0.488500000) (20,0.639300000) (21,0.776700000) (22,0.856200000) (23,0.922600000) (24,0.956100000) (25,0.978900000) (26,0.988200000) (27,0.994700000) (28,0.997700000) (29,0.999100000) (30,0.999800000) (31,0.999700000) (32,1.000000000) (33,1.000000000) (34,1.000000000) (35,1.000000000) (36,1.000000000) (37,1.000000000) (38,1.000000000) (39,1.000000000) (40,1.000000000)};
	% What fraction of n-permutations are 6-universal?
%	\draw [thick] 
%		plot [mark=none] 
%		file {data/Universal-Fractions-6-10000.txt};
	\draw [thick] plot [mark=none] coordinates
		{(0,0.000000000) (1,0.000000000) (2,0.000000000) (3,0.000000000) (4,0.000000000) (5,0.000000000) (6,0.000000000) (7,0.000000000) (8,0.000000000) (9,0.000000000) (10,0.000000000) (11,0.000000000) (12,0.000000000) (13,0.000000000) (14,0.000000000) (15,0.000000000) (16,0.000000000) (17,0.000000000) (18,0.000000000) (19,0.000000000) (20,0.000100000) (21,0.001900000) (22,0.009900000) (23,0.038100000) (24,0.100600000) (25,0.196000000) (26,0.317300000) (27,0.462400000) (28,0.583000000) (29,0.711700000) (30,0.805500000) (31,0.879600000) (32,0.922300000) (33,0.953300000) (34,0.975000000) (35,0.985200000) (36,0.991400000) (37,0.996200000) (38,0.997800000) (39,0.999300000) (40,0.999200000)};
\end{tikzpicture}
\end{center}
}
\caption{The proportion of permutations containing subsequences order-isomorphic to every permutation of length $n$, by length, for $3\le n\le 6$.}
\label{fig-plot-universal}
\end{figure}
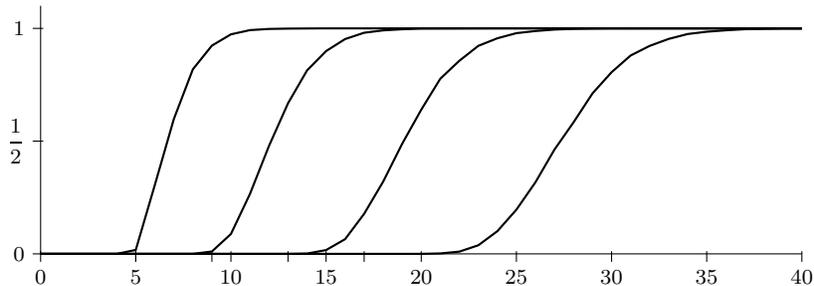

\item At the end of his paper, Arratia~\cite{arratia:on-the-stanley-:} defines $t(n)$ to be the least integer $m$ such that at least half of all permutations of length $m$ contain subsequences order-isomorphic to every permutation of length $n$, and he states that Noga Alon has conjectured that $t(n)$ is asymptotic to $n^2/4$. Figure~\ref{fig-plot-universal} plots the proportions of these permutations of lengths $0\le m\le 40$ for $n=3$, $4$, $5$, and $6$. For $n=3$, we compute these proportions exactly using the formula of Simion and Schmidt~\cite{simion:restricted-perm:} mentioned at the beginning of Section~\ref{sec-subseq-P}, while for $n\ge 4$, these plots are obtained by random sampling to a high level of confidence. This data and further computations suggest the following values of $t(n)$ for $1\le n\le 8$:
	\[
	\begin{array}{cc}
		n&t(n)\\\hline
		1&1\\
		2&3\\
		3&7\\
		4&13\\
		5&20\\
		6&28\\
		7&36\\
		8&48
	\end{array}
	\]
While the first six values of $t(n)$ above might lead the reader to suspect that $t(n)$ is the nearest integer to $\pi n^2/4$, this seems not to hold for $n=7,8$. We leave it to the reader to decide whether these values support or undermine Alon's conjecture that $t(n)\sim n^2/4$.

\item Universal words over $\mathbb{P}$ containing, as subsequences, all permutations of length $n$ from a proper permutation class have also been studied. Bannister, Cheng, Devanny, and Eppstein~\cite{bannister:superpatterns-a:} construct a universal word of length $n^2/4+\Theta(n)$ for the permutations of length $n$ in the class $\Av(132)$, and they show that every \emph{proper} subclass $\C\subsetneq\Av(132)$ has a universal word of length at most $O(n\log^{O(1)} n)$. In~\cite{bannister:superpatterns-a:}, among other results, Bannister, Devanny, and Eppstein find a universal word of length at most $22n^{3/2}+\Theta(n)$ for the class $\Av(321)$. Finally, Albert, Engen, Pantone, and Vatter~\cite{albert:universal-layer:} consider the class of layered permutations, $\Av(231, 312)$. In addition to verifying a conjecture of Gray~\cite{gray:bounds-on-super:}, they show that the length of the shortest universal word over $\mathbb{P}$ containing all layered permutations of length $n$ as subsequences is given \emph{precisely} by $a(0)=0$ and
\[
	a(n) = (n+1)\lceil\log_2 (n+1)\rceil -2^{\lceil \log_2 (n+1)\rceil} + 1
\]
for $n\ge 1$.

\end{enumerate}

\bigskip

\noindent{\bf Acknowledgements.}
We thank Michael Albert, Arnar Arnarson, Robert Brignall, Robin Houston, and Jay Pantone for numerous fruitful discussions that improved this work. We are additionally grateful to Jay Pantone for his assistance in verifying that no permutation of length $16$ or less contains all permutations of length $6$ as subsequences.

\bibliographystyle{acm}
\bibliography{../../../refs}

\end{document}